\documentclass[12pt,epsfig]{amsart}
\usepackage{latexsym}
\usepackage{amsmath,amscd,amssymb}
\usepackage{graphics}
 2

\usepackage[utf8]{inputenc}
\usepackage[T1]{fontenc}

\include{plaquettemacro}
\usepackage{hyperref}
\newtheorem{theorem}{Theorem}[section]
\newtheorem{prop}{Proposition}[section]

\newtheorem{definition}[theorem]{Definition}
\newtheorem{example}[theorem]{Example}

\numberwithin{equation}{section}

\newcommand{\p}{{\partial}}
\title[Interpolating singularity and curve by a generalised maximal surface]{Existence of maximal surface containing given curve and special singularity}
\author{Rukmini Dey}
\address{International Centre for Theoretical sciences,
Bengaluru- 560 089, India.}
\email{rukmini@icts.res.in}
\author{Pradip Kumar}{\thanks{Second and Third authors would like to thank ICTS-TIFR for its hospitality while this work was completed}}
\address{Department of Mathematics, Shiv Nadar University, Dadri 201314, Uttarpradesh, India.}
\email{pmishra.math@gmail.com}
\author{Rahul Kumar Singh}
\address{Department of Mathematics,
  Harish-Chandra Research Institute, HBNI, Chhatnag Road, Jhunsi, Allahabad-211019, Uttarpradesh\\ India}
\email{rahulkumar@hri.res.in; rhlsngh498@gmail.com}
\subjclass[2010]{53A35, 53B30}
\keywords{Maximal surfaces, Bjorling formula, Singularities}
\begin{document}
\maketitle
\begin{abstract}
We give a different formulation for describing maximal surfaces in Lorentz-Minkowski space, $\mathbb{L}^3$, using the identification of $\mathbb L^3$ with $\mathbb C\times \mathbb R$. Further we give a different proof for the singular Bj\"orling problem for the case of closed real analytic null curve. As an application, we show the existence of maximal surface which contains a given curve and has a special singularity.
\end{abstract}
\section{Introduction}
Generalised maximal surfaces in Lorentz-Minkowski space $\mathbb L^3$ are spacelike immersions with zero mean curvature and singularities. In this article, we ask: does there exists a generalised maximal surface containing given curve and having a special singularity.  We start with the following example.

Let $\alpha(\theta)= (-\frac{3}{4} \cos\theta, -\frac{3}{4} \sin \theta, \ln \frac{1}{2})$ be a spacelike closed real analytic curve. This curve lies on elliptic catenoid, a maximal surface, given by  map
$F(x, y)=\left(\frac{x(x^2+y^2-1)}{2(x^2+y^2)}, \frac{y(x^2+y^2-1)}{2(x^2+y^2)}, \ln\sqrt{x^2+y^2}\right)$. We see that
\begin{enumerate}
\item the map $F$ is defined for all $z =x+iy \neq 0$ and has conelike singularity on $|z|=1$.
\item there is a positive real $r_0$, namely $r_0= \frac{1}{2}$ such that $F(|z|=r_0)= \gamma(\frac{1}{2}e^{i\theta}):=\alpha(\theta).$
\end{enumerate}
On the other hand if we take $\beta(\theta)= (e^{i\theta}, 1)$, which is a spacelike closed curve, we will see  (in section 4) that there does not exist any maximal surface $F$ (parametrised by single chart $F$ defined for all $z\neq 0$) and any $r_0\neq 1$ such that $F(r_0e^{i\theta})=\tilde{\beta}(r_0e^{i\theta}):= \beta(\theta)$ and has singularity (not necessarily cone like) at $|z|=1$.

In this article, we will see that if given curve $\gamma$ satisfies some conditions then there exists a generalised maximal surface which has property (1) and (2) as above.

This follows from the solution of  the Bj\"orling problem for maximal surface.  Al\'ias, Chaves, and Mira in \cite{mira_2} solved the Bj\"orling problem for maximal surfaces.  Kim and Yang in \cite{kim} introduced the singular Bj\"orling problem and proved it for the case of real analytic null curve  defined on an open interval. Immediate extension of the singular Bj\"orling problem  and solution for the case of closed curve was discussed by the same authors in \cite{kim}.  We have revisited this problem for the case of closed analytic curve and given a different proof for this.  We believe the technique used in our proof helps to know more about the generalised maximal surface. In  particular it helps  to solve  the problem discussed in the beginning of the introduction.

  When one curve is constant and other curve is nonconstant real analytic and closed, we can ask: can we interpolate both curves by some generalised maximal surface such that the point (corresponding to the constant curve) is a singularity. We answer this question under certain conditions on nonconstant real analytic curve. In this article we did not solve interpolation problem in general, namely the question of existence of a maximal surface interpolating two real analytic curves. This we save for future work.

Article is arranged as follows: In section 2, we discussed the maximal surface and its properties in a modified way.  In section 3, we revisit the singular Bj\"orling problem discussed by Kim and Yang \cite{kim} in a different way. In section 4, we discussed the particular case of the interpolation problem.\\
Our article is self contained and the content of the article is  motivated by (in particular for maximal surface) \cite{neill},\cite{lopez},\cite{fujimori},\cite{kim},\cite{romero} and for minimal surface  by \cite{iwaniec}.
\section{maximal surface}
Vector space $\mathbb R^3$ with the metric $dx^2+dy^2-dt^2$, denoted by $\mathbb L ^3$,  is known as Lorentz-Minkowski space. We identify the vector space structure of $\mathbb L^3$ with $\mathbb C\times \mathbb R$, by $(x,y,t)\to (x+iy, t)$ then the metric is $(dx+idy)(dx-idy)- dt^2$.

\begin{definition}
Let $\Omega\subset\mathbb C$ be a domain and  $F= (u,v,w): \Omega\to \mathbb L^3$ be a nonconstant, smooth harmonic map such that  the coordinate functions $u, v, w$ satisfy the conformality relations (with $z=x+iy$),
\begin{eqnarray}\label{isothermal}
&&u_x^2+v_x^2-w_x^2= u_y^2+v_y^2-w_y^2\\
&&u_xu_y+v_xv_y-w_xw_y=0\nonumber
\end{eqnarray}
and on $\Omega$,  $|\frac{\p u}{\p z}|^2+|\frac{\p v}{\p z}|^2-|\frac{\p w}{\p z}|^2$ does not vanish identically.
$F$ is said to be a Generalised maximal surface.
\end{definition}
Let $ F=(h:=u+iv,w) $, where $ h $ is the complex coordinate of $ F $, the conformality relations \eqref{isothermal} is equivalent to
\begin{equation*}
h_z\overline{{h}_{\bar{z}}}-w_z^2=0.
\end{equation*}
On $\Omega$, nonvanishing of  $|\frac{\p u}{\p z}|^2+|\frac{\p v}{\p z}|^2-|\frac{\p w}{\p z}|^2$ is equivalent to  $|h_z|$ is not identically equal to $|h_{\bar{z}}|$.
In view of the above complex representation, we have an equivalent definition of the generalised maximal surface.
\begin{definition}\label{Def:GeneralisedMaxSurface}  Let $ F=(h,w):\Omega\to \mathbb C\times \mathbb R $ be a smooth map such that  $h_{z\overline{z}}=0$ and $w_{z\overline{z}}=0$ with
$h_z\overline{{h}_{\bar{z}}}-w_z^2=0$
and  ${|h_z|}$ is not identically equal to ${|h_{\bar{z}}|}$. Generalised maximal surface is the equivalence class of map $F$, where equivalence relation is change of conformal parameter.
\end{definition}
\begin{example}[Elliptic catenoid]\label{elliptic_catenoid_1}\normalfont Let $\Omega= \mathbb C- \{0\}$ and $h(z)= \frac{1}{2}\left(z-\frac{1}{\bar{z}}\right)$, $w(z)= \frac{1}{2}\log(z\overline{z})$. Then we define $F:\mathbb C-\{0\}\to \mathbb C\times \mathbb R$, $F(z)= (h(z), w(z))$.
We have $h_z\overline{{h}_{\bar{z}}}-w_z^2=0$ and $h_{z\bar{z}}= w_{z\bar{z}}=0$ for all $z\in \Omega$.
Here $ |h_z|$ is not identically equal to $|h_{\bar{z}}| $. Only on $ |z|=1 $, $ |h_z|=|h_{\bar{z}}|. $
\end{example}
\begin{example}\normalfont
If we take $ h:\mathbb{C}\rightarrow\mathbb{C} $ defined by $ h(z)=\sin z + \sin{\bar{z}}+i0 $ and  $w(z)= \sin z+\sin{\bar{z}}$, then we have $h_z\overline{{h}_{\bar{z}}}-w_z^2=0$ and $h_{z\bar{z}}= w_{z\bar{z}}=0$ for all $z\in \mathbb{C}$, but $ |h_z|$ is identically equal to $|h_{\bar{z}}| $ on whole of $ \mathbb{C} $. Therefore it is not a generalised maximal surface.
\end{example}
Since $ F $ is a generalised maximal surface in isothermal parameters, we have $\langle F_x,F_x\rangle = \langle F_y,F_y\rangle=\eta, \langle F_x,F_y\rangle=0 $, we have
\begin{equation}
ds^2=\eta(z)(dx^2+dy^2)=\eta(z)|dz|^2,\;\text{ where}
\end{equation}
\begin{equation}
\begin{split}
\eta(z) & = \langle F_x,F_x\rangle \\
 & = \langle (h_z + h_{\bar{z}}, w_z + w_{\bar{z}}),(h_z + h_{\bar{z}}, w_z +
 w_{\bar{z}})\rangle \\
 & = (h_z + h_{\bar{z}})\overline{(h_z + h_{\bar{z}})}-(w_z + w_{\bar{z}})^2
 \end{split}
\end{equation}
Now using conformality relation $ h_z\overline{{h}_{\bar{z}}}-w_z^2=0$, we obtain \begin{center}
$\eta(z)=(|h_z|-|h_{\bar{z}}|)^2$.
\end{center}
A point of $ \Omega\subseteq\mathbb{C} $ on which the equation $ |h_z|=|h_{\bar{z}}| $ holds is called a singular point of $ (F,\Omega) $ and set of all \textit{singular points} is called the \textit{singular set} of the maximal surface $(F,\Omega)$.
Based on image of singularity set,  authors in \cite{kim}, \cite{fujimori}, \cite{kobayashi}, \cite{umehara} have discussed various kind of singularities, such as shrinking, curvilinear singularity, cuspidal edges, swallowtails etc. Fern\'{a}ndez, L\'{o}pez, and Souam in \cite{lopez} discussed two type of isolated singularity; branch and special singularity. We also use the name special singularity for the singularity as defined below.
\begin{definition}[Special singularity]
A point $ p $ in $ \mathbb{L}^3 $ is such that $ F(\{|z|=r\})=p $ for some $ r>0 $, then we say that at $ p $  the generalised maximal surface $(F,\Omega)$ has special singularity, if $ |z|=r $ is a subset of the singular set (set of all singular points) of $ (F ,\Omega) $.
\end{definition}
If $(F, \Omega)$  has a special singularity at a point $p$ for $|z|=r$, we often refer to it as $ p $ or $ |z|=r $.

Points where $ |h_z|\neq|h_{\bar{z}}| $ holds are called  regular point of $ (F,\Omega) $ in the sense that at those points of $\Omega$, $F$ will be immersion.
We have following easy observation that  if $ F $ is not an immersion, then in particular $ u_xv_y-u_yv_x=0$. In turn $ u_xv_y-u_yv_x= |h_z|^2-|h_{\bar{z}}|^2 $. Thus $ |h_z|=|h_{\bar{z}}| $.

Conversely, suppose $ |h_z|=|h_{\bar{z}}| $, as $ F=(h,w) $ is a generalized  maximal surface then $ |h_z|=|h_{\bar{z}}| $ corresponds to singular set of the surface.
Indeed, since we have $h_z\overline{{h}_{\bar{z}}}-w_z^2=0$, this imply $ |w_z|^2=|h_z|^2=|h_{\bar{z}}|^2 $. This also gives
\begin{equation}
2(u_xv_y-u_yv_x)=(u_x^2+v_x^2-w_x^2)+(u_y^2+v_y^2-w_y^2)
\end{equation}
As we have $F$ maximal, so  by definition $ F $ is spacelike. Therefore, the vectors $F_x=(u_x, v_x, w_x)$ and $F_y=(u_y,v_y,w_y)$ are spacelike vectors and hence
 \begin{center}
$|F_x|^2=u_x^2+v_x^2-w_x^2\geq0 $
\end{center}
\begin{center}
$|F_y|^2=u_y^2+v_y^2-w_y^2\geq0 .$
\end{center}
therefore we get $ |F_x|^2+|F_y|^2= 0$. This imply $ F_x=F_y=0 $. Thus $ F $ is not an immersion.
Therefore we see that  $F= (h,w):\Omega\to \mathbb L^3$ is a generalised maximal surface.  $F$ is immersion at $p\in \Omega$ if and only if at $p$, $|h_z|\neq |h_{\bar{z}}|.$

With this representation of maximal surface, following \cite{iwaniec}, we have the following:
\begin{prop}\label{complexrepresentaion}
Let $h:\Omega\to \mathbb C$ be the complex coordinate of the isothermal representation of a generalized maximal surface $F=(h,w):\Omega\to \mathbb C\times\mathbb R\simeq \mathbb L^3$. Then on $\Omega\subset \mathbb C$, we can write
    $$w(z)= 2 Re\int_{z_0}^z\sqrt{h_z\overline{{h}_{\bar{z}}}}dz+ w(z_0),$$ where the line integral is along any smooth curve starting from $z_0$ and ending at $z$.
\end{prop}
\begin{proof}
The function $h_z\overline{{h}_{\bar{z}}}$ admits a continuous branch of square root in $\Omega$.  Let $\Gamma$ be a closed curve in $\Omega$.  Consider
    $$2Re\int_\Gamma \sqrt{h_z\overline{{h}_{\bar{z}}}}dz= \int_\Gamma \sqrt{h_z\overline{{h}_{\bar{z}}}}dz+\overline{\int_\Gamma \sqrt{h_z\overline{{h}_{\bar{z}}}}dz}=\int_\Gamma \omega_z dz+ \int_\Gamma\overline{\omega_z}\overline{dz}=\int_\Gamma dw =0.$$
Therefore we have  for every closed curve $\Gamma\subset \Omega$,
   $$Re\int_\Gamma\sqrt{h_z\overline{{h}_{\bar{z}}}} dz=0.$$
This allows us to take $w(z)-w(z_0)= 2Re\int_{z_0}^z\sqrt{h_z\overline{{h}_{\bar{z}}}}dz.$
    This gives
    $$w(z)= 2Re\int_{z_0}^z\sqrt{h_z\overline{{h}_{\bar{z}}}}dz+w(z_0).$$
\end{proof}
The complex coordinate representation (as in definition \eqref{Def:GeneralisedMaxSurface} and in proposition \eqref{complexrepresentaion}) of the generalised maximal surface helps us to construct various examples of maximal surfaces.  In particular if we take any complex harmonic map $h: \Omega\to\mathbb C$ such that $|h_z|$ is not identically same as $|h_{\bar{z}}|$, then the map defined $F: \Omega\to \mathbb L^3$, defined by $F(z)= \left(h(z),2Re\int_{z_0}^z\sqrt{h_z\overline{{h}_{\bar{z}}}}dz\right)$ is a generalised maximal surface.
\begin{example}\normalfont
If we take $ h:\mathbb{C}\rightarrow\mathbb{C} $ defined by $ h(z)=e^{z}+ \bar{z} $. Then $ h_z=e^z, h_{\bar{z}}=1 $ and hence $ |h_z|=|h_{\bar{z}}|=1 $ on imaginary axis. Here $|h_z|$ is not identically equal to $|h_{\bar{z}}|$, by proposition \eqref{complexrepresentaion}, we can determine the third real coordinate $ w$ to make $(h,w)$ a maximal surface.
$$w(z)= 2 Re\int\sqrt{h_z\overline{{h}_{\bar{z}}}}dz=2(e^\frac{z}{2}+e^\frac{\bar{z}}{2}). $$  The map $ F:\mathbb{C}\rightarrow\mathbb{L}^3$ given by $ F(z)=(h(z),w(z)) $ satisfies $ h_z\overline{{h}_{\bar{z}}}-w_z^2=0 $ (conformality relations) and $ h_{z\bar{z}}=0 $, $ w_{z\bar{z}}=0 $(harmonicity) and hence defines a generalized maximal surface.
\end{example}
\begin{example}\normalfont
If we take $h(z)= \frac{1}{2}\left(z-\frac{1}{\bar{z}}\right)$, by proposition \eqref{complexrepresentaion}, we get $w(z)= \frac{1}{2}\log(z\bar{z})$. Then $ F(z)=(h(z),w(z)) $ defines what is known as a elliptic catenoid which is a generalised maximal surface with singular set the unit circle $ \{|z|=1\} $.
\end{example}


Normal vector at a regular point of a generalised maximal surface can be given by a map $N: \Omega\to \mathbb H^2:=\{(x,y,t)\in\mathbb L^3: x^2+y^2-t^2=-1\}$,
    \begin{align}\label{normalh}
    N(z)= \frac{F_x\times F_y}{|F_x\times F_y|}=\left(\frac{2\sqrt{h_zh_{\bar{z}}}}{|h_{\bar{z}}|-|h_z|}, \frac{|h_{\bar{z}}|+|h_z|}{|h_{\bar{z}}|-|h_z|}\right).
    \end{align}

For a generalised maximal surface $(F,\Omega)$, $\Omega$ has three parts $\mathcal{A}:=\{z:|h_{\bar{z}}|<|h_z|\}$,  $\mathcal{B}:=\{z:|h_{\bar{z}}|=|h_z|\}$ and $\mathcal{C}:=\{z:|h_{\bar{z}}|>|h_z|\}$.
As we defined earlier,  $\mathcal{B}$ denotes the singular set  of $(F,\Omega)$. The Gauss map at regular points (that is on $\mathcal{A}$ and on $\mathcal{C})$ is obtained by stereographic projection of $N$ as in \eqref{normalh} from $\mathbb H^2$ to $\mathbb C$. It is   given by
\begin{enumerate}\label{gaussmap}
\item $\nu(z)= \sqrt{\dfrac{h_z}{\overline{h_{\bar{z}}}}}$ on $ \mathcal{A} $
\item $\nu(z)= -\sqrt{\dfrac{h_{\bar{z}}}{\overline{h_z}}}$ on $ \mathcal{C}. $ \end{enumerate}

\section{singular Bj\"orling problem}
Let \begin{align} \label{singulardata}
 \gamma(e^{i\theta})=\left((\gamma_1+i\gamma_2)(e^{i\theta}),\gamma_3(e^{i\theta})\right)
\end{align}
\begin{align*}
L(e^{i\theta})=\left((L_1+iL_2)(e^{i\theta}),L_3(e^{i\theta})\right)
\end{align*}
be such that $\langle \gamma',L\rangle=0 $, where $\gamma$ is a null real analytic closed curve  and $L$ is a null real  analytic vector field and that atleast one of $\gamma'$ and $L$ is not identically zero, $\gamma $ and $L$ both are defined over $S^1$. The above data is known as singular Bj\"orling data for closed curve. Kim and Yang in \cite{kim} studied the singular Bj\"orling problem in detail. In this section we will discuss the same problem for the closed curve from a different point of view.  The singular Bj\"orling problem asks for the  existence of  a generalised maximal surface $$ F=(h,w):A(r,R)\rightarrow \mathbb{L}^3 $$ such that $ F(e^{i\theta})=\gamma(e^{i\theta})~~~\text{and}~~~ \left.\dfrac{\partial F}{\partial\rho}\right\vert_{e^{i\theta}}=(h_{\rho}(e^{i\theta}),w_{\rho}(e^{i\theta}))=L(e^{i\theta})$  with singular set  atleast $\{|z|=1\}$.

For the existence of maximal surface, having prescribed data as above, we will be looking for complex harmonic maps $ h $ and $ w $ on some annulus $ A(r,R) $, $r<1<R$  such that they satisfy
\begin{enumerate}
 \item $h_z\overline{h_{\bar{z}}}-w_z^2\equiv 0 $
 \item $|h_z|=|h_{\bar{z}}|~~~\text{on}~~~ z=e^{i\theta}$
 \item $ |h_z|-|h_{\bar{z}}|$ is not identically zero on $A(r,R).$
 \end{enumerate}

We have the following relation between first order partial differentials in system $ (z,\bar{z}) $ to the first order partial differential in system $ (\rho,\theta) $; where $ z=\rho e^{i\theta} $:
 \begin{equation}\label{hzandhzbar}
 h_{z}=\dfrac{1}{2}\left(h_{\rho}-\dfrac{i}{\rho}h_{\theta}\right)e^{{-i}{\theta}}~~~\text{and}~~~ h_{\bar{z}}=\dfrac{1}{2}\left(h_{\rho}+\dfrac{i}{\rho}h_{\theta}\right)e^{{i}{\theta}}.
\end{equation}
Here we have given $ (h_{\rho}, w_{\rho})=(L_1+iL_2, L_3) $ and $ (h_{\theta},w_{\theta})=({\gamma_1}'+i{\gamma_2}',{\gamma_3}') $ on the unit circle. On $ \{|z|=1\}$, we define the maps $g_1$ and $g_2$ as
\begin{eqnarray}\label{expressionfor_g_whengammaorLvanish}
 &&g_1(e^{i\theta})=\sqrt{\dfrac{L_1+iL_2}{L_1-iL_2}};\; \text{ if } \gamma' \text{ vanishes identically.}\\
&&  g_2(e^{i\theta})=\sqrt{\dfrac{{\gamma_1}'+i{\gamma_2}'}{{\gamma_1}'-i{\gamma_2}'}};\; \text { if }  L \text{ vanishes identically.}
\end{eqnarray}
Since $\gamma'$ and $L$ are dependent (being null vector field and perpendicular), if both $\gamma'$ and $L$ do not vanish identically, we have $g_1(e^{i\theta})= g_2(e^{i\theta})$. Therefore we get a well defined map  $g$ on $ S^1 $ given by $g_1$ and $g_2$ as above.

If there exists a generalised maximal surface $(h,w) $ for the given Bj\"orling data, then analytic extension of $g$ agrees with $\nu=\sqrt{\dfrac{h_z}{\overline{h_{\bar{z}}}}}$ on $\mathcal{A}$ (that is at those points of the domain where $|h_{\bar{z}}|<|h_z|$).  Similarly there is a real analytic function on $S^1$ whose analytic extension matches with $\nu$ on $\mathcal{C}$.

We have the following existence theorem.
\begin{theorem}\label{singularBjorlingProblemThm}
Given a real analytic null closed curve $\gamma:S^1\to \mathbb L^3$ and a null vector field  $L:S^1 \to \mathbb L^3$ such that $ \langle{\gamma'},L\rangle=0 $; atleast one of $\gamma'$ and $L$ do not vanish identically.  If $|g(z)|$ ($ g(z)$ is analytic extension of $ g(e^{i\theta}))$ is not identically equal to $1$, then there exists a unique generalised maximal surface $ F:=(h,w) $ defined on some annulus $ A(r,R):=\{z:0<r<|z|<R\}; r<1<R , $ such that
\begin{enumerate}
\item $ F(e^{i\theta})=(h(e^{i\theta}),w(e^{i\theta}))=\gamma(e^{i\theta}). $
\item $ \left.\dfrac{\partial F}{\partial\rho}\right\vert_{e^{i\theta}}=(h_{\rho}(e^{i\theta}),w_{\rho}(e^{i\theta}))=L(e^{i\theta}).$
\end{enumerate}
 with singular set  atleast $\{|z|=1\}$.
\end{theorem}

\begin{proof}
We will prove this theorem in two steps
\begin{enumerate}
\item We show the existence of generalised maximal surface containing the given singular B\j\"orling data.
\item Next we show that the determined generalized maximal surface will have singularity set atleast $ \{|z|=1\}. $
\end{enumerate}
In the step $1$, we find a complex harmonic function $h$ and a real harmonic function $ w $ defined on some annulus $ A(r,R), $ and show that $h_z\overline{h_{\bar{z}}}-w_z^2\equiv 0$. Any harmonic function over some annulus $A(r,R)$ has the following form
\begin{equation}\label{harmonic_exprsn_of_h(z)}
h(z)= \sum_{-\infty}^\infty a_n z^n+\frac{b_n}{\bar{z}^{n}}+ c \ln |z|^2.
\end{equation}

Therefore, in $(\rho, \theta)$ coordinates, on the unit circle,
\begin{equation}\label{htheta} h_\theta(e^{i\theta})= i\sum_{-\infty}^\infty n (a_n+b_n)e^{in\theta} \end{equation}
\begin{equation}\label{hrho} h_\rho(e^{i\theta})= \sum_{-\infty}^\infty n (a_n-b_n)e^{in\theta} +c\end{equation}

From the given data, $h(e^{i\theta})= \gamma_1(\theta)+ i\gamma_2(\theta),$  we know  the left hand side of the equation \eqref{htheta} and as $\gamma$ is analytic, $h_\theta(e^{i\theta})$ is analytic so the series (in equation \eqref{htheta}) in the right hand side converges.

Next we equate \begin{equation}\label{hr} h_\rho(e^{i\theta})= L_1(\theta)+ iL_2(\theta)\end{equation} as above, as $L_1+iL_2$ is  analytic,
$ h_{\rho} $ is analytic and hence the series in equation \eqref{hrho} converges. We have $n(a_n+b_n) $ as the fourier coefficients of $h_\theta$ in equation \eqref{htheta} for all $n$, and those for $ h_{\rho} $ are $ n(a_n-b_n) $ in equation \eqref{hrho}, for all $ n $. Therefore we can solve for $a_n$, $b_n$ and $c$ uniquely and hence we have  determined $h(z)$ such that $h$ is harmonic.\\  
In the same way, the harmonic function $ w(z) $ can be determined, because we have given $ w(e^{i\theta}) $ and $ w_{\rho}(e^{i\theta}) $.


Now we will show $ h_z\overline{{h}_{\bar{z}}}-w_z^2=0 $ on unit circle with given data. Indeed,

\begin{equation}\label{hbarzbar}
\overline{h_{\bar{z}}}=\dfrac{1}{2}\left(\overline{h_{\rho}}-\dfrac{i}{\rho}\overline{h_{\theta}}\right)e^{{-i}{\theta}}
\end{equation}
\begin{equation}\label{wz}
 w_{z}=\dfrac{1}{2}\left(w_{\rho}-\dfrac{i}{\rho}w_{\theta}\right)e^{{-i}{\theta}}
\end{equation}
 On unit circle we have
 \begin{align*}
 h_{z}\overline{h_{\bar{z}}}&=\dfrac{1}{4}(h_{\rho}-ih_{\theta})(\overline{h_{\rho}}-i\overline{h_{\theta}})e^{-2i\theta}\\
                            &=\dfrac{1}{4}(L_1^2+L_2^2-{\gamma'}_1^2-{\gamma'}_2^2-i(L_1+iL_2)({\gamma'}_1-i{\gamma'}_2)-i({\gamma'}_1+i{\gamma'}_2)(L_1-iL_2))e^{-2i\theta}\\
 \end{align*}
 As $ L $ and $ {\gamma}' $ are null vector fields we have $ L_1^2+L_2^2=L_3^2 $ and $ {\gamma'}_1^2+{\gamma'}_2^2={\gamma'}_3^2 $, using these identities in above equation we get
 \begin{align}\label{barhz}
 h_{z}\overline{h_{\bar{z}}}&=\dfrac{1}{4}(L_3^2-{\gamma'}_3^2-2iL_3{\gamma'}_3)e^{-2i\theta}
 \end{align}
 Next
 \begin{align}\label{wz2}
 w_{z}^2=\dfrac{1}{4}(w_{\rho}-iw_{\theta})^2e^{-2i\theta}
        &=\dfrac{1}{4}(L_3^2-{\gamma'}_3^2-2iL_3{\gamma'}_3)e^{-2i\theta}
 \end{align}
 From equation \eqref{barhz} and equation \eqref{wz2} we see that $ h_z\overline{h_{\bar{z}}}-w_z^2=0 $ on the unit circle. As $ h $ and $ w $ are harmonic functions on an annulus $ A(r,R) $, the function $ h_z\overline{h_{\bar{z}}}-w_z^2  $ is complex analytic on $ A(r,R) $ which contains the unit circle and hence $ h_z\overline{h_{\bar{z}}}-w_z^2\equiv0  $ on annulus.\\ 
 As we have given in hypothesis, that the analytic extension $ g(z) $ of $ g(e^{i\theta}) $ is such that $ |g(z)| $ is not identically $ 1 $, which is, equivalent to saying $ |h_z| $ is not identically equal to $ |h_{\bar{z}}| $. Hence we have found the unique generalised  maximal surface $ F:=(h,w) $.

 Now in this last step, we are going to show that the singular set of the generalised maximal surface $ F:=(h,w) $ contains atleast $\{|z|=1\} $.

Since $ L_3{\gamma'}_3=L_1{\gamma'}_1+L_2{\gamma'}_2 $ and $\gamma'$ and $L$ are null vector field, we have $ L_1{\gamma'}_2=L_2{\gamma'}_1 $.

By the expression \eqref{hzandhzbar}, on the unit circle we have
 \begin{align*}
  |h_z|=\dfrac{1}{2}|h_{\rho}-ih_{\theta}|
               =|L_{1}+{\gamma'}_{2}+i(L_{2}-{\gamma'}_{1})|\end{align*}
  \begin{align}\label{R}
  |h_z|^2=\dfrac{L_1^2+L_2^2+{\gamma'}_1^2+{\gamma'}_2^2}{4}+\dfrac{L_1{\gamma'}_2-L_2{\gamma'}_1}{2}
  \end{align}
  Similarly,
  \begin{align}\label{S}
   |h_{\bar{z}}|^2=\dfrac{L_1^2+L_2^2+{\gamma'}_1^2+{\gamma'}_2^2}{4}-\dfrac{L_1{\gamma'}_2-L_2{\gamma'}_1}{2}
  \end{align}
  Now subtracting equation \eqref{S} from \eqref{R} will give
  \begin{align*}
  |h_z|^2-|h_{\bar{z}}|^2=L_1{\gamma'}_2-L_2{\gamma'}_1=0.
  \end{align*}
  Thus $ |h_z|=|h_{\bar{z}}| $ on unit circle, this proves that our unique generalised maximal surface $ F:=(h,w) $ will have singularity set atleast $ \{|z|=1\} .$
 \end{proof}
 \begin{example}\normalfont If $ \gamma(\theta)=(c,c,c) $, a  constant curve then for any non vanishing null vector field $ L(\theta) $ there exists a generalised maximal surface containing the constant curve as singularity.  We give an example illustrating this and the proof of the above theorem.  When $ L(\theta)=(e^{i\theta},1)=(h_{\rho},w_{\rho}) $ and $ {\gamma'}(\theta)=(0+i0,0)=(h_{\theta},w_{\theta}) $, we will get a generalised maximal surface known as elliptic catenoid.
 Recall the expressions \eqref{htheta} and \eqref{hrho}, from these we have
 \begin{align*}
 0= i\sum_{-\infty}^\infty n (a_n+b_n)e^{in\theta} \text{ and }
 \end{align*}
 \begin{align*}
 e^{i\theta}= \sum_{-\infty}^\infty n (a_n-b_n)e^{in\theta} +c
 \end{align*}
 This gives $ a_1-b_1=1$ and  $a_1+b_1=0$ which
 imply $ a_n=0 $, $ b_n=0,\;\forall\; n\neq 1$ and $ c=0 $ and hence from the formula \eqref{harmonic_exprsn_of_h(z)} we get $
 h(z)=\dfrac{1}{2}\left(z-\dfrac{1}{\bar{z}}\right)$.
 To obtain $ w(z)$, we repeat the same step as in the case of obtaining $ h(z) $, because here we know $ w_{\rho}=1 $ and $ w_{\theta}=0 $, from this we get $c=1$ and $ a_n=b_n=0, \;\forall\; n $. This gives $
 w(z)=\dfrac{1}{2}log(z\bar{z}).$
 Expressions $(h,w)$ together represents an elliptic catenoid.
 \end{example}

\section{Existence of maximal surface containing a presecribed curve and special singularity}

We start with an example to explain the problem and possible solution. Let $ \tilde{\gamma}(\theta)=(c_1e^{i\theta},c_2), $ $c_1$ and $c_2$ be some constants. We will see that if we take $c_1=c_2=1$, i.e. $ \tilde{\gamma}(\theta)=(e^{i\theta},1), $ then  there does not exists any positive real $r_0\neq 1$ and generalised maximal surface $F$  as in the Definition \eqref{Def:GeneralisedMaxSurface}, defined on some annulus having $|z|=1$  such that  $ F(r_0e^{i\theta})=(e^{i\theta}, 1) $  and $ F $ restricted to unit circle has a special singularity.

While, in particular, if we take $c_1= -\frac{3}{4}$ and $c_2= \ln \frac{1}{2}$, then for $r_0=\frac{1}{2}$, there is a generalised maximal surface $F: \mathbb C-\{0\}\to \mathbb L^3$ such that $F(r_0e^{i\theta})= \gamma(r_0e^{i\theta}):= \tilde{\gamma}(\theta)$, maximal surface is the elliptic catenoid discussed in the example \eqref{elliptic_catenoid_1}.

Now below we will verify the above facts in detail. We are looking for the generalised maximal surface $F$ such that
 \begin{align}\label{initialdata}
 F(r_0{e^{i\theta}})=(c_1{e^{i\theta}}, c_2); r_{0}\neq 1, c_1, c_2 ~~~\text{are constants} ~~~\text{and}~~~ F(e^{i\theta})=(0,0,0).
 \end{align}
 Also on $|z|=1$, $F$ admits singularity.

 Suppose if we can find such a maximal surface $ F(z)=(h(z),w(z)) $ which satisfy the initial data  given in \eqref{initialdata}, then $h$ for the maximal surface  over an annulus is of the form as in equation \eqref{harmonic_exprsn_of_h(z)}, and similarly for $w$. The initial condition $F(e^{i\theta})= (0+0i, 0)$ will give us
 \begin{align}\label{an}
 a_{n}+b_{n}=0;\;\; \forall\; n
\end{align}
 and the condition  $ F(r_0{e^{i\theta}})=(c_1{e^{i\theta}}, c_2)$
 \begin{align}\label{eq1}
 a_{n}r_{0}^n+\dfrac{b_{n}}{r_{0}}=0 \Rightarrow  a_{n}=b_{n}=0;\; \forall\; n \neq1,0.
 \end{align}
  \begin{align}\label{eq2}
  a_{0}+b_{0}+c\log r_{0}=0 \Rightarrow c=0 ~~~\text{if}~~~ r_0\neq 1.
  \end{align}
We use \eqref{an}, \eqref{eq1} and \eqref{eq2} to get
  \begin{align}
  a_1=\dfrac{r_0}{r_0^2-1}c_1 ~~~\text{and}~~~ b_1=-a_1,
  \end{align}
  hence
  \begin{align}\label{hz}
  h(z)=\dfrac{r_0c_1}{r_0^2-1}\left(z-\dfrac{1}{\bar{z}}\right)
  \end{align}
  Similarly for $ w(z) $ using initial conditions \eqref{initialdata} we get
  \begin{align}
  c_n+d_n=0;\;\; \forall\; n
  \end{align}
  \begin{align}
 c_{n}r_{0}^n+\dfrac{d_{n}}{r_{0}^n}=0 \Rightarrow  c_{n}=d_{n}=0;\;\; \forall n \neq1,0.
 \end{align}
  \begin{align}
  c_{0}+d_{0}+d\log r_{0}=c_2 \Rightarrow d=\dfrac{c_2}{\log r_0} ~~~\text{if}~~~ r_0\neq 1.
  \end{align}

  \begin{align}\label{eq4}
  w(z)=\left(\dfrac{c_2}{2\log r_0}\right)\log z\bar{z}
  \end{align}
  Now in order to have $ F(z)=(h(z),w(z)) $ as the generalised  maximal surface, $ h $ and $ w $ have to satisfy the conditions given in Definition \eqref{Def:GeneralisedMaxSurface}. The relation $h_z\overline{h_{\bar{z}}}-w_z^2\equiv 0$ gives us a relation between $ c_1, c_2 $ and $ r_0 $ as follows
  \begin{align}\label{constantintheexample}
  \dfrac{c_1r_0}{r_0^2-1}=\dfrac{c_2}{2\log r_0}
  \end{align}
and we see for any set of constants $c_1, c_2, r_0$, satisfies above relation, $|h_z|$ is not identically same as $|h_{\overline{z}}|$. Therefore if we have constants $(c_1, c_2, r_0\neq 1)$ such that they satisfies \eqref{constantintheexample}, then there is a generalised maximal surface satisfying initial data \eqref{initialdata} and having singularity on $|z|=1$.\\
 Moreover, we see that, for the spacelike closed curve $\tilde{\gamma}(\theta)=(e^{i\theta}, 1) $, $ c_1=c_2=1 $, mentioned in the beginning of this section, the equation \eqref{constantintheexample} has no solution for any $r_0$. Therefore, there does not exists any generalised maximal surface $F$ such that
 \begin{align}
 F(r_0{e^{i\theta}})=(e^{i\theta}, 1); r_{0}\neq 1 ~~~\text{,}~~~ F(e^{i\theta})=(0,0,0)
 \end{align}
 and $ F $ restricted to unit circle has a special singularity.

 In general we can ask the following:
Given a real analytic curve $\tilde{\gamma}(\theta)$, does there exists $F: A(r,R) \to \mathbb L^3$, a generalised maximal surface and $r_0\neq 1$ such that $F(r_0e^{i\theta})=\tilde{\gamma}(\theta)$ and $F$ has a special singularity at $|z|=1$.

For a curve $\tilde{\gamma}(\theta)=\gamma(r_0e^{i\theta})= (f(r_0e^{i\theta}), g(r_0e^{i\theta}))$, $r_0\neq 1$  $\left(\text{where }f(r_0e^{i\theta})\in \mathbb C, g(r_0e^{i\theta})\in \mathbb R\right)$, we define the following modified Fourier coefficients  of $f$ and $g$ as
\begin{align}\label{c_nd_d}
c=\dfrac{1}{2\pi\log r_0}\int_{-\pi}^{\pi}f(r_0e^{i\theta})d\theta;\;\;\;
d=\dfrac{1}{2\pi\log r_0}\int_{-\pi}^{\pi}g(r_0e^{i\theta})d\theta,
\end{align}
$\text{for}\;\; n\neq 0;$
\begin{align}\label{cn_and_dn}
c_n=\dfrac{r_0^n}{2\pi(r_0^{2n}-1)}\int_{-\pi}^{\pi}f(r_0e^{i\theta})e^{-in\theta}d\theta;\;\;d_n=\dfrac{r_0^n}{2\pi(r_0^{2n}-1)}\int_{-\pi}^{\pi}g(r_0e^{i\theta})e^{-in\theta}d\theta.
\end{align}

We see that if $\tilde{\gamma}$ is real analytic (since $c_n, d_n, c_{-n},$ and $d_{-n}$ all converges to 0),  $\limsup |c_{-n}|^{\frac{1}{n}}=0,$ $\limsup|c_n|^{\frac{1}{n}}=0,$ $\limsup |d_{-n}|^{\frac{1}{n}}=0$ and $\limsup |d_{n}|^{\frac{1}{n}}=0$, therefore the following two series convreges for all $|z|\neq 0$,
\begin{align}\label{h_zintermsofc_n}
h(z)=\sum_{-\infty}^{\infty}c_n\left(z^n-\dfrac{1}{\bar{z}^n}\right)+ c\log |z|.
\end{align}

\begin{align}\label{w_zintermofd_n}
w(z)=\sum_{-\infty}^{\infty}d_n\left(z^n-\dfrac{1}{\bar{z}^n}\right)+ d\log |z|,
\end{align}

Now we state the following theorem  which is an application to the theorem \eqref{singularBjorlingProblemThm}.
\begin{theorem}\label{theoreminterpolation}
Let $\tilde{\gamma}(\theta)$ be a nonconstant closed real analytic spacelike curve. Then there exists $s_0\neq 1$ and  a generalised maximal surface $F: \mathbb C-\{0\}\to \mathbb L^3$ such that $F(s_0e^{i\theta}):=\tilde{\gamma}(\theta)$ and having a special singularity at $(0,0,0)\in\mathbb L^3$ if and only if   there exists $r_0\neq 1$ and constants $c, c_n's, d, d_n's$ for the curve $\gamma(r_0e^{i\theta}):=\tilde{\gamma}(\theta)=(f(r_0e^{i\theta}),g(r_0e^{i\theta}))$, as in equations  \eqref{c_nd_d},\eqref{cn_and_dn} which satisfy the relations:
\begin{eqnarray}\label{series condition1}
&&\forall\; k\neq 0;\;\;\sum_{-\infty}^\infty 4n(n-k)(c_n\bar{c}_{n-k}-d_nd_{n-k})+ 2k(c_k\bar{c}-c\bar{c}_{-k}-2d_kd)=0\\
\label{series condition2}&&\text { and } \;\sum 4n^2(c_n\bar{c}_n-d_n^2)+ c\bar{c}-d^2=0.
\end{eqnarray}

\end{theorem}
\begin{proof} We start proving only if part. Assume that the constants $c, c_n's, d, d_n's$ satisfies the conditions \eqref{series condition1} and \eqref{series condition2} for the curve $\gamma(r_0e^{i\theta})$. We claim that $h$ and $w$ given by equation \eqref{h_zintermsofc_n}, \eqref{w_zintermofd_n}  is the generalised maximal surface satisfying given data. We see that $h(|z|=1)=0;\;\;w(|z|=1)=0$ and $\gamma(r_0e^{i\theta})=(h(r_0e^{i\theta}), w(r_0e^{i\theta}))= (f(r_0e^{i\theta}), g(r_0e^{i\theta})).$
From equations \eqref{h_zintermsofc_n} and \eqref{w_zintermofd_n}, we have
$$h_\rho(e^{i\theta})= \sum_{-\infty}^\infty 2nc_ne^{in\theta} +c;\text{ and }$$
$$w_\rho(e^{i\theta})= \sum_{-\infty}^\infty 2nd_ne^{in\theta} +d$$

\begin{eqnarray*}h_\rho(e^{i\theta}).\bar{h}_\rho(e^{i\theta})&=& \left(\sum_{-\infty}^\infty2nc_ne^{in\theta}+c\right)\left(\sum_{-\infty}^\infty2n\bar{c}_ne^{-in\theta}+\bar{c}\right)\\
&=&\sum_{k=-\infty}^\infty\left(\sum_{n=-\infty}^\infty 4n(n-k)c_n\bar{c}_{n-k}+2k(c_k\bar{c}-c\bar{c}_{-k})\right)e^{ik\theta}+c\bar{c}
\end{eqnarray*}
Similarly we have
$$w_\rho^2(e^{i\theta})= \sum_{k=-\infty}^\infty\left(\sum_{n=-\infty}^\infty 4n(n-k)d_nd_{n-k}+4kd_kd\right)e^{ik\theta}+d^2$$

All the series above converges absolutely as $f$ and $g$ are real analytic functions. The series conditions on the constants given in the theorem assures that $h_\rho\overline{h_\rho}-w_\rho^2=0$ for $z= e^{i\theta}$, that is to say that $(h_\rho, w_\rho)$ is a null vector field along $|z|=1$.

By the singular Bj\"orling problem for closed curve $\alpha(e^{i\theta})= (0,0,0)$ and $L(e^{i\theta})= (h_\rho(e^{\theta}), w_\rho(e^{\theta}))$ we have a unique maximal surface $(h', w')$ on some $A(r,R)$, $r<1<R$ by theorem \eqref{singularBjorlingProblemThm}. On $A(r, R)$ we have (by uniqueness of $(h,w)$ and $(h', w')$ on $A(r,R)$),
$$h_z\overline{h_{\bar{z}}}-w_z^2= h'_z\overline{h'_{\bar{z}}}-{w'_z}^2\equiv 0$$
But this being a complex analytic function on $\mathbb C-\{0\}$, $h_z\overline{h_{\bar{z}}}-w_z^2\equiv 0$.   Also as $F(r_0e^{i\theta})$ is a spacelike curve, this gives that $|h_z|$ is not identically equal to $|h_{\bar{z}}|$, and it proves existence of the required generalised maximal surface. Here we can take $ s_0=r_0 $.

Now other way, if the generalized maximal surface $F$ is given such that $F(s_0e^{i\theta})= \gamma(\theta)$ and $F(|z|=1)= (0,0,0)$, then $F$ has to be of the form $(h,w)$ as given in the equations \eqref{h_zintermsofc_n} and \eqref{w_zintermofd_n} with $c, c_n's, d, d_n's$ as in \eqref{c_nd_d} and \eqref{cn_and_dn} with $r_0= s_0$. In this form prescribing singularity set as $|z|=1$ is same as asking for the vector $(h_\rho(e^{i\theta}), w_\rho(e^{i\theta}))$ is a null vector which gives series conditions as in \eqref{series condition1} and \eqref{series condition2}.
\end{proof}

For a given spacelike closed curve, $r_0(\neq 1)$ may not exist as in the above theorem unless it satisfies those series conditions and if such an $ r_0 $ exists, it need not be unique. For instance, we have seen that for the curve $\tilde{\gamma}(\theta)= (e^{i\theta},1)$, there does not exists a generalised maximal surface and  $r_0\neq 1$ such that $F(r_0e^{i\theta})= \tilde{\gamma}(\theta); F(e^{i\theta})= (0,0,0)$ with singular set atleast $|z|=1$. 
But if we do some small perturbations of this curve $\tilde{\gamma}$, i.e., for $\epsilon>0$,  let $F(r_0^{\epsilon}e^{i\theta})=\tilde{\gamma}_{\epsilon}(\theta)= ((1-\epsilon)e^{i\theta},1)$, compare this with (\ref{initialdata}), then $ \frac{c_1}{c_2}=1-\epsilon $, and from the equation (\ref{constantintheexample}) we see that there are two choice of $r_0^{\epsilon}$ for fixed $\epsilon$. Also, we can see that as $\epsilon\to 0$,$\tilde{\gamma}_{\epsilon}\to \tilde{\gamma} $ and $r_0^{\epsilon}\to 1$.

In the above theorem, fixing the special singularity at $(0,0,0)$ and asking for the existence of a generalised maximal surface is not necessary, we may ask for any point $(x_1, x_2, x_3)\in \mathbb L^3$ as the special singularity corresponding to $|z|=1$.  But then accordingly the expression of $h$ and $w$ as in \eqref{h_zintermsofc_n} and \eqref{w_zintermofd_n} will change and the new series conditions (for e.g. \eqref{series condition1} and \eqref{series condition2}) will be found by posing condition that new $(h_\rho, w_\rho)$ is null vector along $|z|=1$.   We believe it is not the statement but the proof of the theorem that gives a handy way to check existence of the generalised maximal surface for a given closed spacelike curve.

\begin{example}\normalfont  We have seen that for $\tilde{\gamma}(\theta)= \left(-\frac{3}{4}e^{i\theta}, \ln\frac{1}{2}\right),$ if we take $r_0=\frac{1}{2}$ then the constants as in equations \eqref{c_nd_d}, \eqref{cn_and_dn} are as follows
$c_1=\frac{1}{2}\;, d=1$ and for all $n\neq 1$, $c_n's=0,\; d_n=0$ and $c=0, d_1=0$ and these constants satisfies the series conditions as in equations \eqref{series condition1} and \eqref{series condition2}.  Therefore there exists a generalised maximal surface which is  given by expression of $h$ and $w$ as in example \eqref{elliptic_catenoid_1} having special singularity.
\end{example}

\begin{example}\normalfont
Consider the curve
\begin{align}
\tilde{\gamma}(\theta)=(a_3e^{3i\theta}+a_1e^{-i\theta}, b_2e^{2i\theta}+b_2e^{-2i\theta}).
\end{align}
Below we will analyze for given constants $a_1, a_3, $ and $b_2$, does there exists $r_0\neq 1$ and the generalised maximal surface as in theorem above (\eqref{theoreminterpolation}).

Recall the formulas \eqref{c_nd_d} and \eqref{cn_and_dn}, for $ f(r_0e^{i\theta})=a_3e^{3i\theta}+a_1e^{-i\theta} $and $g(r_0e^{i\theta})=b_2e^{2i\theta}+b_2e^{-2i\theta}  $, then we have $ c_0=0 , c=0$, for  $n\neq -1, 3$; $  c_n=0 $ and
\begin{align}\label{c's_intheexample}
c_{-1}=\dfrac{r_0^{-1}}{r_0^{-2}-1}a_1 ~~~\text{,}~~~ c_3=\dfrac{r_0^{3}}{r_0^{6}-1}a_3.
\end{align}
Similarly $ d=0 $, for $ n\neq -2, 2;$ $ d_n=0 $ and
\begin{align}\label{d's_intheexample}
d_2=\dfrac{r_0^{2}}{r_0^{4}-1}b_2 ~~~\text{,}~~~ d_{-2}=\dfrac{r_0^{-2}}{r_0^{-4}-1}b_2.
\end{align}
Suppose the constants $ a_1, a_3 $  and $ b_2 $ are such that the curve $ \tilde{\gamma} $ is spacelike, then there exists a generalised maximal surface $ F $ and $ r_0 $ as in  theorem \eqref{theoreminterpolation} if and only if the conditions \eqref{series condition1} and \eqref{series condition2} are satisfied by the constants $ c,c_n's,d,d_n's $. That is to say
 $$\forall\; k\neq 0\; \sum_{n=-2,-1,2,3} 4n(n-k)(c_nc_{n-k}-d_nd_{n-k})=0\;\; \text{ and }\sum_{n=2,-2,-1,3} 4n^2(c_n^2-d_n^2)=0$$ which is equivalent to
\begin{align}
4d_{2}d_{-2}-3c_{3}c_{-1}=0~~~\text{and}~~~c_{-1}^2+9c_3^2= 4(d_2^2+d_{-2}^2).
\end{align}
Therefore for $\gamma(r_0e^{i\theta}):=\tilde{\gamma}(\theta)=(a_3e^{3i\theta}+a_1e^{-i\theta}, b_2e^{2i\theta}+b_2e^{-2i\theta})$  if $a_1, a_3, b_2$  are such that $\tilde{\gamma}$ is spacelike then there exists a maximal surface
$F:\mathbb C-\{0\} \to \mathbb L^3$ such that $F(r_0e^{i\theta})= \tilde{\gamma}(\theta)$ and $F$ has special singularity at $|z|=1$ if and only if
\begin{eqnarray}
\label{exp_1}&&4\dfrac{b_2}{\left(r_0^2-\dfrac{1}{r_0^2}\right)}\dfrac{b_2}{\left(r_0^2-\dfrac{1}{r_0^2}\right)}=\dfrac{3a_3}{\left(r_0^3-\dfrac{1}{r_0^3}\right)}\dfrac{a_1}{\left(r_0-\dfrac{1}{r_0}\right)}\text{ and }\\
\label{exp_2}&&\dfrac{a_1^2}{\left(r_0-\dfrac{1}{r_0}\right)^2}+\dfrac{9a_3^2}{\left(r_0^3-\dfrac{1}{r_0^3}\right)^2}= 2\dfrac{4b_2^2}{\left(r_0^2-\dfrac{1}{r_0^2}\right)^2}.
\end{eqnarray}

In particular, for any given positive real $c\neq 1$, constants $ a_1, a_3 $ and $ b_2 $ as  $a_1= \frac{1}{2}\left(c-\dfrac{1}{c}\right)$,  $a_3= \frac{1}{6}\left(c^3-\dfrac{1}{c^3}\right)$ and $b_2=\frac{1}{4}\left(c^2-\dfrac{1}{c^2}\right)$ satisfy the equations \eqref{exp_1} and \eqref{exp_2} with $r_0=c$.

Also for any positive $c\neq 1$, the curve
$$\tilde{\gamma}(\theta)= \left(\frac{1}{2}\left(c-\dfrac{1}{c}\right)e^{-i\theta}+\frac{1}{6}\left(c^3-\dfrac{1}{c^3}\right)e^{3i\theta}, \frac{1}{2}\left(c^2-\dfrac{1}{c^2}\right)\cos 2\theta\right)$$
is spacelike. Therefore there is a generalised maximal surface as in theorem \eqref{theoreminterpolation} containing the curve $\tilde{\gamma}$ as above with special singularity at $|z|=1$.  The generalised maximal surface is given by $$h(z)= \frac{1}{6}\left(z^3-\frac{1}{\bar{z}^3}\right)+\frac{1}{2}\left(\bar{z}-\frac{1}{z}\right);\;\;w(z)= \frac{1}{4}\left(z^2-\frac{1}{\bar{z}^2}-\frac{1}{z^2}+\bar{z}^2\right).$$
\end{example}

\section{Conclusion}
We have seen  how complex representation of generalised maximal surface $$F(z)= \left(h(z),2Re\int_{z_0}^z\sqrt{h_z\overline{h}_{\overline{z}}}dz\right),$$
helps us to solve the singular Bj\"orling problem (theorem \eqref{singularBjorlingProblemThm}) and a particular type of interpolation problem (theorem \eqref{theoreminterpolation}).  This representation helps to see that  there does not exists any $r_0\neq 1$ such that $F(r_0e^{i\theta})= (e^{i\theta},1)$ with  special singularity on $|z|=1$,  but small perturbation of curve  $(e^{i\theta},1)$ gives the existence of required $r_0$ and $F$.

In general  there may exist a  maximal surface $G$ defined over a Riemann surface such that $G$ may have special singularity at some point and there may be curve on that Riemann surface whose image is arbitrary curve say for example $\tilde{\gamma}=(e^{i\theta}, 1)$.  This is a general problem which still needs to be explored for general curves.

\end{document}